\documentclass[12pt]{amsart}

\usepackage{enumerate}

\usepackage{amsmath}
 \usepackage{amsfonts,amssymb,graphicx}
\usepackage[colorlinks=true,linkcolor=blue,citecolor=blue]{hyperref}

\newtheorem{theorem}{Theorem}
\newtheorem{proposition}{Proposition}
\newtheorem{corollary}{Corollary}
\newtheorem{remark}{Remark}
\newtheorem{lemma}{Lemma}

\newcommand{\Z}[1][]{\ensuremath{{\mathbb{Z}^{#1}} }}
\newcommand{\C}[1][]{\ensuremath{{\mathbb{C}^{#1}} }}
\newcommand{\R}[1][]{\ensuremath{{\mathbb{R}^{#1}} }}
\newcommand{\Q}[1][]{\ensuremath{{\mathbb{Q}^{#1}} }}
\renewcommand{\S}[1][]{\ensuremath{{\mathbb{S}^{#1}} }}
\renewcommand{\H}[1][]{\ensuremath{{\mathbb{H}^{#1}} }}
\def\T{\mathbb{T}}
\def\P{\mathbb{P}}
\def\Re{ \mathrm{Re}\, }
\def\Im{ \mathrm{Im}\, }

\newcommand{\<}{\langle}
\renewcommand{\>}{\rangle}

\newcommand{\pa}{\partial}

\newcommand{\al}{\alpha}

\date{}

\title[Construction of H-minimal Lagrangian submanifolds in $\C^n$]
{Construction of Hamiltonian-minimal Lagrangian submanifolds in complex Euclidean space}

\author{Henri Anciaux}
\address{Henri Anciaux \\
Universidade de S\~ao Paulo \\ 
  IME, Bloco A \\
  1010 Rua do Mat\~ao  \\
Cidade Universit\'aria   \\ 
 05508-090 S\~ao Paulo, BRAZIL 
 \\} \email{henri.anciaux@gmail.com}

\author{Ildefonso Castro}
\address{Departamento de Matem\'{a}ticas \\
Universidad de Ja\'{e}n \\
23071 Ja\'{e}n, SPAIN}
\email{icastro@ujaen.es}
\thanks{Research partially supported by a MEC-Feder grant  MTM2007-61779(Second author)}

\begin{document}
\maketitle

\section{Introduction}

A submanifold $L$ of $\C^{n}$ (or more generally of a symplectic
manifold ${ M}^{2n}$ of real dimension $2n$) is said to be  \em Lagrangian \em if it has
dimension $n$ and if the standard symplectic form $\omega$
vanishes on it. This assumption is equivalent to the fact that the
standard complex structure $J$ maps the tangent bundle of $L$ onto
its normal bundle. A remarkable property of Lagrangian
submanifolds of $\C^{n}$ is that those which are in addition
minimal, (i.e.\ critical points for the volume functional) are in
fact minimizers in their homology class. The reason is that
$\C^{n}$ is endowed with a one-parameter family of calibrations
whose calibrated submanifolds are precisely the  minimal
Lagrangian submanifolds (see \cite{HL}). This property does not extend
to generic K\"{a}hler manifolds, but it does to a certain subclass
of them, called \em Calabi-Yau manifolds \em (see \cite{HL} for a definition). Lagrangian
submanifolds of $\C^{n}$ are locally characterized as being the
gradient graphs $L :=\{ (x + i \, \nabla u(x)), x \in \R^n \} $ of
real-valued functions $u(x).$ In terms of these data, the minimal
submanifold equation is (see \cite{HL}):
$$   \mathrm{Im} \det_{\C} (I + i \, \mbox{Hess}  u)=0.$$
In the case of dimension $2,$ it reduces to the famous \em
Monge-Amp\`ere equation. \em

\medskip

Besides the classical variational problem of minimizing the volume
functional in a homology class (so in particular with respect to
compactly supported variations), there is a natural variational
problem, first introduced by Oh (cf \cite{O}), consisting of minimizing
the volume with respect to Hamiltonian compactly supported
variations. Such variations have the property of preserving the
Lagrangian constraint. We shall say that a Lagrangian submanifold
is \em Hamiltonian-minimal\footnote{Some authors use the terminology \em Hamiltonian stationary \em} \em or \em
H-minimal \em for short if it is a critical point of the volume for
Hamiltonian compactly supported variations. While it is well known
that minimal submanifolds are characterized by the vanishing of
their mean curvature vector $H=\frac{1}{n}\,\mathrm{trace}\,
\sigma$, where $\sigma$ is the second fundamental form of $\Phi $,
it can be proved that a Lagrangian submanifold of $\C^{n}$ is
H-minimal if and only if it satisfies the equation $\mbox{div}
JH=0$, where $\mbox{div}$ denotes the divergence operator.
 In particular, Lagrangian submanifolds with parallel mean
curvature vector are H-minimal.

 The H-minimal equation for a
gradient graph is (cf \cite{SW}):
$$\sum_{j=1}^n \frac{\pa}{\pa x_j} \Delta (\frac{\pa u}{\pa x_j})=0,$$
where $\Delta$ denotes the Laplacian with respect to the induced
metric on $L.$ Denoting partial derivatives by subscripts, the
latter is given in coordinates by
$$ g_{jk} = \delta_{jk}+ \sum_{l=1}^n u_{jl} u_{kl}.$$
In particular the H-minimal equation is of fourth order and its
linearization is the bilaplacian equation.

\medskip

 One physical motivation for studying H-minimal Lagrangian
 submanifold is given by the model of incompressible elasticity:
  a diffeomorphism $\big(X(x,y),Y(x,y) \big)$ between two open subsets $U$ and $V$ of
 $\C$ is incompressible, i.e.
 $ X_x Y_y - X_y Y_x=1,$ if and only if the graph
 $$ L := \big\{ \big(x+iy, X(x,y) - iY(x,y)\big) , x+iy \in U   \big\}$$ is
 Lagrangian. Moreover,  the diffeomorphism $(X,Y)$ minimizes the functional
 $\int_U \sqrt{1+ |\nabla X|^2+ |\nabla Y|^2}$ among incompressible diffeomorphisms
  if and only if $L$ is
 H-minimal (cf \cite{W}).

\medskip

A very important object attached to an oriented Lagrangian submanifold $L$ of
$\C^{n}$ is its \em Lagrangian angle function, \em which is
defined to be the argument of the evaluation of the complex volume
form $ dz_1\wedge\dots\wedge dz_n $. In other words,
if $(e_1,...,e_n)$ is a tangent frame of $L$, its Lagrangian angle function is given by
$$ \beta= \arg dz_1\wedge\dots\wedge dz_n (e_1,...,e_n).$$
The Lagrangian angle if fundamental in the study of variational
problems since it is related to the mean curvature vector of $L$
by the formula $n H = J \nabla \beta$, where $\nabla$ is the gradient for the
induced metric. It follows that $L$ is
minimal if and only if its Lagrangian angle  is equal to a
constant $\beta_0.$ In this case $L$ is in addition calibrated by
the $n$-form $ \Re (e^{-i\beta_0}dz_1\wedge\dots\wedge dz_n).$ Moreover, we can
deduce a characterization of those Lagrangian submanifolds which
are H-minimal in terms of $\beta$: by the equation $\mbox{div} JH=-
\frac{1}{n} \Delta \beta,$ we deduce that a Lagrangian submanifold
is H-minimal if and only if its Lagrangian angle is harmonic for
the induced metric.

\medskip

  Until recently only very
 simple examples where known beyond  the Cartesian products
 of $n$ circles  $ \S^1(r_1) \times ... \times \S^1(r_n) \subset \C^n .$
After the first description of non-trivial H-minimal tori in
\cite{CU1}, the H-minimal cones and tori of $\C^2$ where classified
respectively in \cite{SW} and \cite{HR2}. Then more examples were
discovered in $\C^2$ (cf \cite{A1}), in $\C^n$ (cf \cite{ACR}, \cite{M1}, \cite{CLU})
and in $\C\P^n$ (cf \cite{M1}, \cite{M2}, \cite{CLU}). The goal of the present
paper is to describe in a synthetic way a variety of examples of
 H-minimal Lagrangian immersions in  $\C^{n},$ combining in several
ways curves in two dimensional space forms and Legendrian
immersions in odd dimensional spheres.
  These constructions  appear
 to be generalizations of examples
 already discussed in \cite{ACR}, \cite{AR} and \cite{CLU}. In every case we shall characterize
 the cases in which the submanifold is in addition minimal, or have
 parallel mean curvature vector. We shall also pay attention to the cases in which we
 get compact examples.
Finally we point out
   that the importance of H-minimal
  Lagrangian submanifolds of $\C^n$ is emphasized by a recent
  work of Joyce, Lee and Schoen, where H-minimal Lagrangian
  submanifolds are constructed in arbitrary symplectic manifolds,
  starting from a H-minimal Lagrangian submanifold of $\C^n$
  satisfying some property of rigidity
(cf \cite{JLS}).

\medskip

 The paper is organized as follows: the first section gives some background
  about curves and Legendrian submanifolds in odd-dimensional spheres;
  we shall see in particular that the geometry of Legendrian submanifolds is
  very similar to the one of Lagrangian submanifolds.
 The next sections are devoted to several constructions of
 H-minimal Lagrangian submanifolds using, respectively, $n$ planar curves
 (Section 3), a planar
 curve and a Legendrian immersion (Section 4), a Lagrangian
 surface and two Legendrian immersions (Section 5).

\section{Background material}
\subsection{Legendrian immersions}
Let $\C^{n}=\{(z_1,\dots,z_{n} ), \, \, z_j\in\C, \, 1 \leq j \leq n\}$
be the complex Euclidean space of dimension $n$ endowed with the
bilinear product
\[
(z,w)=\sum_{j=1}^n z_j\bar{w}_j, \quad\forall z,w\in\C^n.
\]
Then $\langle .,. \rangle=\Re (.,.)$ is the Euclidean metric  of
$\C^n$ and $\omega=-\Im (.,.)$ the Kaehler $2$-form on $\C^n$.
Both forms are related by the formula
 $\omega=\langle J.,.\rangle$.

 The symplectic form $\omega$ is exact and one  primitive of it is the the
Liouville $1$-form of  $\Lambda$ defined by $
2\Lambda(v)=\langle v,Jz\rangle$, for all $v\in T_z\C^n$,
$z\in\C^n$.

Next we consider the unit sphere  $\S^{2n-1}:=\{z \in \C^n , \langle z,z \rangle =1 \}$
 and we still
 denote by $\Lambda $ the restriction to $\S^{2n-1}$ of
the Liouville $1$-form of $\C^{n}$. Hence $\Lambda $ is the contact
$1$-form of the canonical Sasakian structure on the sphere
$\S^{2n-1}$. An immersion $\psi$  of
an $(n-1)$-dimensional manifold $N$ into $\S^{2n-1}$ is said to be {\em Legendrian}
if $\psi^* \Lambda \equiv 0$. When it is the case $\psi $ is
isotropic in $\C^{n}$, i.e. $\psi^* \omega \equiv 0$ and, in
particular, the normal bundle of $N$ admits the following
decomposition $T^\perp N = J (TN) \oplus {\rm span \, } \{ J\psi
\} $. In other words  $\psi $ is horizontal with respect to the
Hopf fibration $\Pi:\S^{2n-1} \rightarrow \C\P^{n-1}$, where
$\C\P^{n-1} $ denotes the complex projective space with constant
holomorphic sectional curvature 4. Hence $\Pi \,  \circ \, \psi $
 is a Lagrangian immersion of $N$ into $\C\P^{n-1}$ and the two
 immersions $\psi$ and $\Pi \, \circ \, \psi$ induce the same
metric  on $N$. Conversely, given a Lagrangian immersion ${\phi}$
of a $(n-1)$-dimensional manifold $N$ into $\C\P^{n-1}$, there exists a Legendrian
immersion $\tilde{\psi}$ from the universal covering $\tilde{N}$ of $N$ into
$\S^{2n-1}$ such that $ \phi \, \circ \, \tilde{\Pi}= \Pi \, \circ \, \tilde{\psi},$
where $\tilde{\Pi}$ is the canonical projection $\tilde{N} \to N.$
The immersion $\tilde{\psi}$, which is unique up to a rotation of the form
$e^{i\theta} Id \subset U(n),$
is called the \em Legendrian lift \em of $\phi.$

Next we define $\Omega $ to be the complex $(n-1)$-form on $\S^{2n-1}$ given by
$$
\Omega_z(v_1,\dots,v_{n-1})=\det_{\C} \, \{ z,v_1,\dots,v_{n-1} \} .
$$
Given  a Legendrian  immersion $\psi$ of a
manifold $N$ into $\S^{2n-1}$,  $\psi^* \Omega $ is a complex $(n-1)$-form on
$N$. Suppose that our Legendrian submanifold $N$ is oriented. We define
the  map $\beta_\psi$ by
$$e^{i\beta_\psi (x)}=(\psi^* \Omega)_x (e_1,\dots,e_{n-1}),$$
where $\{ e_1,\dots, e_{n-1} \}$ is an oriented orthonormal frame
in $T_x N.$ This $\R / 2\pi \Z$-valued map is well defined and does not depend on the choice of the frame
$\{ e_1,\dots, e_{n-1} \}.$ In \cite{CLU} it is called the {\em Legendrian
angle} map of $\psi $ and it is proved that
\begin{equation} \label{gradient2}
J\nabla \beta_\psi = (n-1) H_\psi,
\end{equation}
where $\nabla $ is the gradient with respect to the induced metric in $N$ and $H_\psi $ is the mean
curvature vector of $\psi $. Hence a Legendrian
immersion $\psi$ of an oriented manifold
$N$ in $\S^{2n-1}$ is minimal, i.e. $H_\psi \equiv 0$, if and only if the
Legendrian angle map $\beta_\psi $ of $\psi $ is constant.

In analogy with the notion of H-minimality for Lagrangian
submanifolds,  the notion of C-minimality was  introduced in \cite{CLU}
as follows: a Legendrian submanifold is said to be {\em
contact-minimal} (or briefly {\em C-minimal}) if it is a  critical
point of the volume functional with respect to compactly supported
variations which preserve the contact form. It is easy to
prove that a Legendrian immersion $\psi$ is C-minimal
 if and only if $ {\rm div} JH_\psi \equiv 0$,
where div is the divergence operator in $N$. We refer to \cite{CLU} for
the proof of this formula and  further details. In particular,
minimal Legendrian submanifolds and Legendrian submanifolds with
parallel mean curvature vector are C-minimal. A consequence of
Equation (\ref{gradient2}) is that a Legendrian immersion $\psi$ of an
 oriented manifold
$N$  into $\S^{2n-1}$ is C-minimal  if and only if the Legendrian angle $\beta_\psi $
of $\psi $ is a harmonic map, i.e.\ $\Delta \beta_\psi \equiv 0$, where
$\Delta $ is the Laplacian for the induced metric.

There is a close relationship between between minimal and
C-minimal Legendrian submanifolds in odd dimensional spheres and
Lagrangian submanifolds in complex projective spaces. More
precisely, it is proved in \cite{CLU}  that a Legendrian immersion $\psi $ is
minimal (resp.\ C-minimal) in $\S^{2n-1}$ if and only if the
Lagrangian immersion $\Pi \circ \psi$ is minimal (resp.\ H-minimal) in $\C\P
^{n-1}$.

\subsection{Planar curves}
If $\alpha : I \rightarrow \C^*$ is a non-vanishing complex function on an interval $I$ of $\Bbb R,$
the argument of $\alpha $ is the $\R / 2\pi \Z$-valued map given by $\alpha = |\alpha| e^{i\,\arg\alpha}$. It is
easy to check that
\begin{equation}\label{arg1}
(\arg \alpha)' = \frac{\langle \alpha ' , J\alpha \rangle}{|\alpha|^2},
\end{equation}
and so we deduce that
\begin{equation}\label{arg2}
(\arg \alpha ')'  =|\alpha ' | \kappa_\alpha ,
\end{equation}
where $\kappa_\alpha$ is the curvature of $\alpha$.

\subsection{Legendrian curves versus spherical and hyperbolic curves}
Let $\S^3$ and $\H^3_1$ denote the
unit hypersphere  and the unit anti de Sitter space in
${\C}^2$, given respectively by
$$\S^3=\left\{ (z,w)\in {\C}^2,  |z|^2+|w|^2=1 \right\}
,\;\;\H^3_1=\left\{ (z,w)\in {\C}^2, |z|^2-|w|^2=-1
   \right\} .$$
   Let $\gamma \subset \S^3$ and $\alpha \subset \H^3_1$ two Legendrian curves,
both  parametrized by arclength. Then they satisfy
the following relations:
\[
\begin{array}{c}
|\gamma_1|^2+|\gamma_2|^2=1,\hskip.2in
|\gamma'_1|^2+|\gamma'_2|^2=1,\;\;\, \gamma'_1
\bar\gamma_1+\gamma'_2\bar\gamma_2=0,\;\;
    \\
|\alpha_1|^2-|\alpha_2|^2=-1,\;\;
|\alpha'_1|^2-|\alpha'_2|^2=1,\;\; \alpha'_1
\bar\alpha_1-\alpha'_2\bar\alpha_2=0.
\end{array}
\]

\medskip

We now detail the special form  taken by the Hopf projection $\Pi$
in the case $n=2,$ as well as the analogous projection $\H^3_1 \to
\H^2(-1/2).$

Let  $\S^2(1/2):=\{ (x_1 + i x_2,x_3)\in \C \times \R,  x_1^2+x_2^2+x_3^2=1/4
\}$ which is the 2-sphere with radius $1/2$ in $\R^3$. The Hopf fibration $\Pi:\S^3\to
\S^2(1/2)\equiv \C \P^1(4)$ is given by
\[
\Pi(z,w)= \frac{1}{2} \left( 2z\bar w, |z|^2-|w|^2 \right),\;\; (z,w)\in \S^3\subset {\C}^2.
\]
For each Legendrian curve $\gamma(s)$ in $\S^3$, the
projection $\xi=\Pi\circ \gamma$ is a curve in $\S^2(1/2)$.
Conversely,   each curve $\xi$ in $\S^2(1/2)$ gives rise to a
horizontal lift $\tilde{\xi}$ in $\S^3$ via $\Pi$ which is unique
up to a rotation  $e^{i\theta} Id \subset U(2) ,\theta \in {\R}$.

Since the Hopf fibration $\Pi $ is a Riemannian submersion, each
 unit speed Legendrian curve $\gamma $ in $\S^3$
 is projected onto   a unit speed curve $\xi$ in  $\S^2(1/2)$
 with the same curvature function.  In addition:
\begin{align}\label{flasCP}
|\gamma_1|^2=\frac{1}{2}+\xi_3, \, \,
\langle \gamma_1',J\gamma_1\rangle=(\xi \times \xi')_3,
\end{align}
where $\times $ denotes the cross product in ${\R}^3$
and $(\xi \times \xi')_3$ is the third coordinate of $\xi \times
\xi'$ in the 3-space ${\R}^3$ containing $\S^2(1/2)$.

Similarly, let  $\H^2(-1/2)=\{ (x_1 + i x_2,x_3)\in \C \times \R,
x_1^2+x_2^2-x_3^2=-1/4, \, x_3 \geq 1/2  \}$ which is the model of
the real hyperbolic plane of curvature $-4$. The Hopf fibration
$\Pi:\H^3_1\to \H^2(-1/2)\equiv \C\H^1(-4)$ is then given by
\[
\Pi(z,w)= \frac{1}{2} \left( 2z\bar w, |z|^2+|w|^2 \right), \;\;
(z,w)\in \H^3_1\subset \C^2. \] Given a Legendrian curve
$\alpha(t)$ in $\H^3_1$, the projection $\eta=\Pi\circ
\alpha$ is a curve in $\H^2(-1/2)$. Conversely,  each curve $\eta$
in $\H^2(-1/2)$ gives rise to a horizontal lift $\tilde{\eta}$ in
$\H^3_1$ via $\Pi $ which is unique up to a rotation of the form
$e^{i\theta}Id \subset U(2), \theta \in \R$.

Similarly, if $\alpha $ is a unit speed Legendrian curve in
$\H^3_1$, then the projection $\eta $ is also a unit speed curve
in $\H^2(-1/2)$ with the same curvature function. It follows that
\begin{align}\label{flasCH}
|\alpha_1|^2=-\frac{1}{2}+\eta_3, \, \,
\langle \alpha_1',J\alpha_1\rangle=(\eta \times \eta')_3.
\end{align}

\section{Product of planar curves}

The simplest way to obtain Lagrangian submanifolds in $\C^n$ is to
take the Cartesian product of $n$ planar curves.
\begin{proposition}\label{product}
Let $\alpha_j$  be $n$ planar regular curves, and denote by $s_j$,
$1 \leq j \leq n$, the arclength parameter of $\al_j$. Then the
product immersion
$$\Phi(s_1, ... , s_n)=(\alpha_1(s_1), ... , \alpha_n(s_n))$$ is a
flat Lagrangian immersion whose Lagrangian angle map is given by
 $$ \beta_\Phi(s_1, ... , s_n) = \sum_{j=1}^n \arg \alpha'_j (s_j) .$$
In addition,
\begin{equation}\label{Hessproduct}
D_{\pa_{s_j}} \nabla \beta_\Phi =  \frac{d\kappa_j}{ds_j} \pa_{s_j},
\end{equation}
where $\kappa_j$ denotes the curvature of the curve $\alpha_j$, $1
\leq j \leq n$, and $D$ is the Levi-Civita connection of the induced metric,
and
\begin{equation}\label{Deltaproduct}
\Delta \beta_\Phi = \sum_{j=1}^n \frac{d\kappa_j}{ds_j}.
\end{equation}
\end{proposition}

\begin{figure}[ht] \begin{center} \includegraphics[scale=0.5]{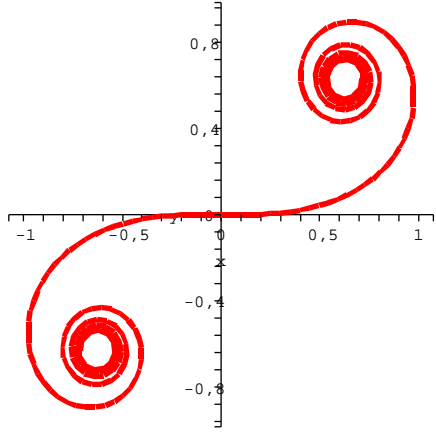}
   \caption{A Cornu spiral (cf Section 3) }
 \end{center}  \end{figure}
   
\begin{proof} Since the induced metric by the immersion $\Phi$ is flat,
using Equation (\ref{arg2}) it is straightforward  to obtain
(\ref{Hessproduct}) and hence (\ref{Deltaproduct}).
\end{proof}

Using (\ref{arg2}) again, we first observe that $\beta_\Phi$ is
constant, i.e.\ $\Phi$ is minimal, if and only if each curve
$\alpha_j$ is a line. In this case $\Phi $ is totally geodesic.

Moreover, $\Phi$ has parallel mean curvature if and only if
each curve $\al_j$ has constant curvature, i.e.\ is a straight line or a circle.
Finally,
$\Delta \beta_\Phi = 0$, i.e.\ $\Phi $ is H-minimal if
and only if the curvature of each curve $\alpha_j$ is a linear
function of its arclength parameter, $\kappa_j(s_j)=\lambda_j s_j
+ \mu_j $, and $\sum_{i=1}^n \lambda_j =0$. Such planar curves are
either curves of constant curvature (when $\lambda_j =0$) or  special
curves known in the literature as \em Cornu spirals \em or \em
clotoids \em (when $\lambda_j \neq 0$). A Cornu spiral $\alpha $
of parameter $\lambda$  can be parametrized, up to
congruences, by
 $$ \alpha (s) = \left( \int_0^{s} \cos(\lambda t^2 /2) dt ,\int_0^{s} \sin (\lambda t^2 / 2) dt  \right). $$
 The Cornu spirals are bounded but have infinite length. They have
 two ends that converges to two points of the plane (see Figure 1).
In conclusion:
\begin{corollary}\label{Cor:Ex0}
A product immersion of $n$ planar curves $\alpha_j$, $1 \leq j
\leq n$ have parallel mean curvature if and only if
each curve $\alpha_j$
is either a circle or  a straight line and
 H-minimal if and only if $(i)$ each curve $\alpha_j$
is either a circle, a line, or a Cornu spiral of parametre
$\lambda_j$, and $(ii)$  $\sum_{j=1}^n \lambda_j=0,$ where we set
$\lambda_j=0$ if $\al_j$ is a circle or a line.
\end{corollary}

\begin{remark} \em We can generalize this construction by considering the product of $n$ Lagrangian immersions into
$\C^{m_j},$ $1 \leq j \leq n,$ resulting in a Lagrangian immersion into $\C^N,$ where $N=\sum_{j=1}^n m_j .$ In the two
following sections, we shall consider more elaborate constructions. \em
\end{remark}

\section{Construction with a planar curve and a Legendrian immersion}
\label{s:Ex1}
\subsection{The construction}
We  now give a precise description of the geometry of a family of
Lagrangian submanifolds in $\C^n$ which was first introduced in
\cite{RU}:
\begin{theorem}\label{Th:Ex1}
Let $\alpha:I\rightarrow\C^*$ be a regular curve, with arclength
parametre $s$   and $\psi$ a
Legendrian immersion of an orientable manifold $N$ into $\S^{2n-1}.$
 Then the map
$$ \begin{array}{lccc} \Phi :
 &  I \times N &  \longrightarrow & \C^{n}              \\
&  (s,x) & \longmapsto & \al(s) \psi(x),
   \end{array}$$
is a Lagrangian immersion in $\C^n$ with induced metric
\[
 \bar{g} = ds^2 + |\alpha|^2 g,
\]
where $g$ is the induced metric on $N$, and Lagrangian angle map
\begin{equation} \label{Beta1}
\beta_\Phi(s,x) =G_{\!\al}(s)+ \beta_\psi(x) ,
\end{equation}
where $G_{\!\al}:= \arg \alpha' + (n-1) \arg \alpha$ and $\beta_\psi$
is the Legendrian angle map of $\psi$. In addition, the gradient of $\beta_\Phi$ is given by the formula:
\begin{equation}\label{Nabla1}
\nabla_{\!\bar{g}} \beta_\Phi=\left(G_{\!\al}' \partial_s,
\frac{1}{|\alpha|^{2}} \nabla_{\!g} \beta_\psi \right).
\end{equation}
Moreover, denoting by $\bar{D}$ (resp. by $D$) the Levi-Civita connection of $\bar{g}$ (resp. of $g$), we have
\begin{equation}\label{Hess11}
\bar{D}_{(\pa_s,0)} \nabla_{\!\bar{g}} \beta_\Phi =
  \left( G''_{\!\al} \pa_s, - \frac{\<\al',\al \>}{|\al|^4} \nabla_{\!g} \beta_\psi \right)
\end{equation}
and
\begin{equation}\label{Hess21}
\bar{D}_{(0,X)} \nabla_{\!\bar{g}} \beta_\Phi =
  \left( - \frac{\<\al',\al \>}{|\al|^2} d\beta_\psi(X) \pa_s,
   \frac{1}{|\al|^2} D_{X} \nabla_{\!g} \beta_\psi+ G_{\!\al}' \frac{\<\al',\al \>}{|\al|^2} X \right),
\end{equation}
where $X$ is a vector field on $N.$

Finally, the Laplacian of $\beta_\Phi$ is given by the formula:
\begin{equation}\label{Delta1}
\Delta_{\bar{g}} \beta_\Phi = \frac{1}{|\alpha|^{n-1}}
\frac{d}{ds} \left(|\al|^{n-1} G'_{\!\al} \right)
    + \frac{1}{|\alpha|^2}\Delta_g \beta_\psi.
\end{equation}
\end{theorem}
\begin{proof}

We consider  local coordinates $(x_1, ... , x_{n-1})$ on $N$,
so that we get local coordinates  $(x_0=s, x_1, ... , x_{n-1})$ on
$ I \times N.$ We calculate the first derivatives of the
immersion:

$$ \frac{\pa \Phi}{\pa x_0}= \al' \psi, \quad \quad
 \frac{\pa \Phi}{\pa x_j}= \al \frac{\pa \psi}{\pa x_j}, \quad 1 \leq j \leq n-1.$$
Calculating the Hermitian products of pairs of such first
derivatives and using the Legendrian property of $\psi$ shows that
the immersion $\Phi$ is Lagrangian. Moreover, denoting by
$(g_{jk})_{1 \leq j,k \leq n-1}$ the coefficients of the metric
$g$, it follows that the coefficients of the metric $\bar{g}$ are
given by
 $$\bar{g}_{jk}= |\al|^2 g_{jk} \quad \bar{g}_{0j}=0 \quad  \bar{g}_{00}=1, $$
and the inverse matrix of $\bar g$ by
 $$\bar{g}^{jk}= |\al|^{-2} g^{jk} \quad \bar{g}^{0j}=0 \quad  \bar{g}^{00}=1. $$
In particular we obtain the formula
$$ \bar{g} = ds^2 + |\alpha|^2 g.$$
Next we compute the Lagrangian angle of $\Phi$:
$$\beta_\Phi=\arg \det_{\C}
\left( \frac{\pa \Phi}{\pa x_0}, ...,  \frac{\pa \Phi}{\pa
x_{n-1}} \right)$$
$$ = \arg \left( \al' \al^{n-1} \det_{\C} \left(\psi,  \frac{\pa \psi}{\pa
x_{1}},..., \frac{\pa \psi}{\pa x_{n-1}} \right) \right) $$
$$= \arg
\al' + (n-1) \arg \al + \beta_{\psi}.$$

We deduce the following expression of the gradient of
$\beta_\Phi$:
$$\nabla_{\!\bar{g}} \beta_\Phi= \sum_{a,b=0}^{n-1}
   \bar{g}^{ab}
               \frac{\pa \beta_\Phi}{\pa x_b}  \pa_{x_a}=\bar{g}^{00}
                G'_{\!\al}  (\pa_s,0)
               +\frac{1}{|\alpha|^{2}} \sum_{j,k=1}^{n-1}
   {g}^{jk}     \frac{\pa \beta_\psi}{\pa x_k}  (0,\pa_{x_j})$$
$$ =G'_{\!\al}(\pa_s,0)+  \frac{1}{|\al|^{2}}(0,\nabla_{\!g} \beta_{\psi}).$$

Next, a straightforward computation gives
$$\bar{D}_{\pa_s} \pa_s = 0   \quad \quad
\bar{D}_{\pa_{x_j}} \pa_s =\bar{D}_{\pa_s} \pa_{x_j} =\frac{\<\al',\al \>}{|\al|^2} \pa_{x_j}$$
and
$$\bar{D}_{\pa_{x_j}} \pa_{x_k} =-\<\al',\al\> g_{jk}\pa_s + D_{\pa_{x_j}} \pa_{x_k}.$$
It follows that
$$ \bar{D}_{\pa_s} \nabla_{\!\bar{g}} \beta_\Phi =
  \left(G_{\!\al}'' \pa_s, 0 \right)$$
  and
  $$ \bar{D}_{\pa_{x_j}} \nabla_{\!\bar{g}} \beta_\Phi =
  \left( - \frac{\<\al',\al \>}{|\al|^2}\frac{\pa \beta_\psi}{\pa x_j} \pa_s,
   \frac{1}{|\al|^2} D_{\pa_{x_j}} \nabla_{\!g} \beta_\psi+ G_{\!\al}' \frac{\<\al',\al \>}{|\al|^2} \pa_{x_j} \right).$$
  This implies Equation (\ref{Hess21}) and
it remains to compute the Laplacian of $\beta_\Phi$:
$$\Delta_{\bar{g}} \beta_\Phi =\sum_{j,k=0}^{n-1} \bar{g} (\bar{D}_{\pa_{x_j}} \nabla_{\!\bar{g}} \beta_\Phi, \pa_{x_k})\bar{g}^{jk}$$
$$ = G_{\!\al}'' + \sum_{j,k=1}^{n-1} g \left( \frac{1}{|\al|^2} D_{\pa_{x_j}} \nabla_{\!g} \beta_\psi, \pa_{x_k} \right)g^{jk}
+ \sum_{j,k=1}^{n-1} G_{\!\al}' \frac{\<\al',\al \>}{|\al|^2} g(\pa_{x_j},\pa_{x_k})g^{jk}$$
$$= G_{\!\al}''  + \frac{1}{|\al|^2}\Delta_g \beta_\psi+ (n-1)G_{\!\al}' \frac{\<\al',\al \>}{|\al|^2}.$$
Finally, we observe that
$$\frac{1}{|\alpha|^{n-1}}
\frac{d}{ds} \left(|\al|^{n-1} G_{\!\al}' \right)= G_{\!\al}''+ (n-1)G_{\!\al}' \frac{\<\al',\al \>}{|\al|^2}$$
and the proof is complete. \end{proof}

\begin{remark}
{\rm If $\alpha $ is a straight line passing through the origin,
the immersion becomes $\Phi : \R \times N \longrightarrow \C^{n}
$, $(s,x) \to s\psi(x)$. Hence we obtain the {\em cone} $ C(\psi)$
with link $\psi $. It is clear that $C(\psi)$ is a Lagrangian
immersion with a singularity at $s=0$. In the general case $\Phi$
has singularities at the points $(s,x)\in I\times N$ where
$\alpha$ vanishes. Since in the case of a cone
$\beta_{C(\psi)}=\beta_\psi$, we deduce $C(\psi )$ is minimal
(resp.\ H-minimal) if and only if $\psi$ is minimal (resp.\
C-minimal); this result was used in \cite{H1} and \cite{H2}.

If we take $\alpha(s)=e^{is}$,
Theorem \ref{Th:Ex1} gives a family of submanifolds which
  generalizes to higher dimension  the Hopf cylinders
  considered in \cite{P}.

If $\psi$ is chosen to be the totally geodesic Legendrian
embedding
 $\psi (x)=x$ of $\S^{n-1}$ into $\S^{2n-1}$,
 we are in the case of the Lagrangian submanifolds
 which are invariant under the standard action of $SO(n)$ on $\C^n$ (see \cite{HL}, \cite{ACR}).}
\end{remark}

The quantity $A_\al :=|\al|^{n-1}G'_{\!\al}$ appearing in the
expressions of the gradient and the Laplacian of $\beta_\Phi$ enjoys
a geometric interpretation. In fact, a simple calculation using
Equation (\ref{arg2}) yields
$$A_\al=|\alpha|^{2n-2}\kappa_{\alpha^n}.$$
where $\kappa_{\alpha^n}$ is the curvature of the planar curve $\al^n.$
From Equation (\ref{Nabla1}) and the formula $nH=J \nabla \beta,$ it
follows that the immersion $\Phi $ constructed in Theorem
\ref{Th:Ex1} is minimal if and only if $\psi$ is minimal and
$\alpha^n$ has curvature zero. This fact has been used in \cite{CU2},
\cite{H1} or \cite{J1}. If $\alpha $ is not a straight line, it can be
parametrized by $\alpha_c(t)=\sqrt[2n]{t^2+c^2}\,
e^{i\frac{\arctan t}{n \, c}}$ and the corresponding minimal
Lagrangian submanifolds were constructed  in \cite{CU2}, (Remark 1), \cite{H1},
(Theorem A) and \cite{J1} (Theorem 6.4). We refer to \cite{CU3} for the
description of some other examples of minimal Lagrangian
submanifolds in $\C^n$ using this method.

Next we give a characterization of those immersions described in
 Theorem \ref{Th:Ex1}
which have parallel mean curvature vector or are compact and H-minimal:

\begin{corollary}\label{Cor:Ex1}
The Lagrangian immersion
$$ \begin{array}{lccc} \Phi :
 &  I \times N &\to& \C^{n}              \\
&  (s,x) & \mapsto & \al(s) \psi(x),
   \end{array}$$
 where $\al$ is a planar curve which does not vanish  and $\psi$ is a
Legendrian immersion of an orientable manifold $N$ into $\S^{2n-1}$,
 has parallel mean curvature vector if and only if
\begin{itemize}
    \item[-] either $(i)$ it is minimal; in this case the curve $\al$ is such that the quantity
$A_\al:=|\alpha|^{2n-2}\kappa_{\alpha^n},$ where $\kappa_{\alpha^n}$ is the curvature of $\al^n$,
vanishes and the immersion
$\psi$ is minimal;
\item[-] or $(ii)$ the curve $\al$ is a circle centered at the origin
and the immersion $\psi$ has parallel mean curvature vector.
\end{itemize}
 Moreover, if $N$ is compact, the immersion
$\Phi$ is H-minimal if and only if the curve $\alpha$ is such that 
$A_\al$  is
constant and the immersion $\psi$ is  C-minimal.
\end{corollary}

\begin{proof}
The first claim follows easily from Equations (\ref{Hess11}) and (\ref{Hess21}). To prove the second claim, we observe that if 
 $\Phi$ is H-minimal,
i.e. $\Delta_{\bar{g}} \beta_\Phi$ vanishes, then
$\Delta_g \beta_\psi$ must be constant by Equation (\ref{Delta1}). Moreover, if $N$ is compact,
  $\int_N \Delta_g \beta_\psi$ vanishes so 
this constant must be zero (so in particular $\psi$ is C-minimal). Thus, by Equation (\ref{Delta1}) again, $A_\al$ must be constant.
\end{proof}

In the next two sections we shall describe in greater detail
respectively the planar curves $\al$ and the Legendrian immersions
$\psi$ satisfying the conditions given in Corollary \ref{Cor:Ex1} to get
 H-minimal Lagrangian submanifolds in
$\C^n$.

\subsection{Planar curves satisfying
$A_\al=|\alpha|^{2n-2}\kappa_{\alpha^n}$ is a non null constant} The
simplest  curves verifying this condition are the circles centered
at the origin, $\alpha (s)=R e^{is/R}$, $R>0$. By Corollary 2, it
provides examples of Lagrangian H-minimal immersions of $\S^1
\times N^{n-1}$ in $\C^n$.
 In particular,  taking
  $\psi (x)=x$, $x\in \S^{2n-1},$ we recover Example 2.10 in \cite{JLS}.

Since the curvature of a curve is changed under a scaling according
to the law
$\kappa_{\lambda\alpha}=\kappa_\alpha / \lambda $,
we deduce that $A_{\lambda \alpha}=\lambda^{n-2} A_\alpha$. Thus, when
$n=2$ this quantity $|\alpha|^{2}\kappa_{\alpha^2}$ is invariant under dilations; on the contrary, when $n>2$
we can normalize $A_\alpha =1$ by rescaling the curve. A qualitative study
of the curves $\alpha$ solutions of
$|\alpha|^{2n-2}\kappa_{\alpha^n}=constant$ has been done for $n=2$
in \cite{AR}, Section 4.1 p.\ 15  and the case  $n >2$ has been treated
in \cite{ACR}, Section 5, p.\ 1204. In both cases the purpose was the
classification of H-minimal Lagrangian submanifolds in $\C^n$
foliated by $(n-1)$-dimensional spheres.
We give here a brief description of these curves (see Figures 2,3,4 and 5):

\begin{figure}[h] \begin{center} \includegraphics[scale=0.5]{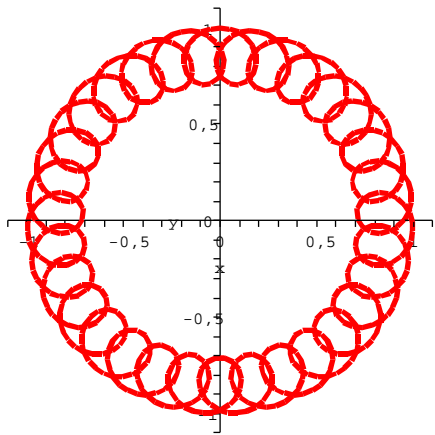}
   \caption{A closed curve $\al_p$ with $A_\al=5$, $n=2$}
\end{center}  \end{figure}

\begin{figure}[h] \begin{center} \includegraphics[scale=0.5]{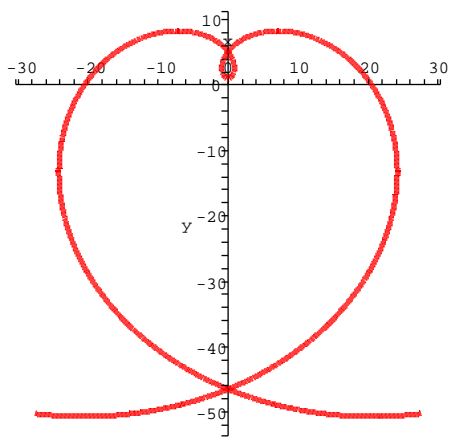}
   \caption{An unbounded curve with $A_\al=1$, $n=2$}
\end{center}  \end{figure}

\noindent {\bf Case $n=2$.} Beyond the circles centered at the
origin, the solutions of the equation
$|\alpha|^{2}\kappa_{\alpha^2} = constant$ belong to one of the
following families:

\begin{itemize}
\item[-]  A two-parameter family of non-embedded
curves, including a countable family of closed curves that we shall denote by $\al_p, p \in \Bbb N.$ The other
ones are not properly embedded;
\item[-]
A one-parameter family of unbounded, non-embedded curves.
\end{itemize}

\begin{figure} \begin{center} \includegraphics[scale=0.5]{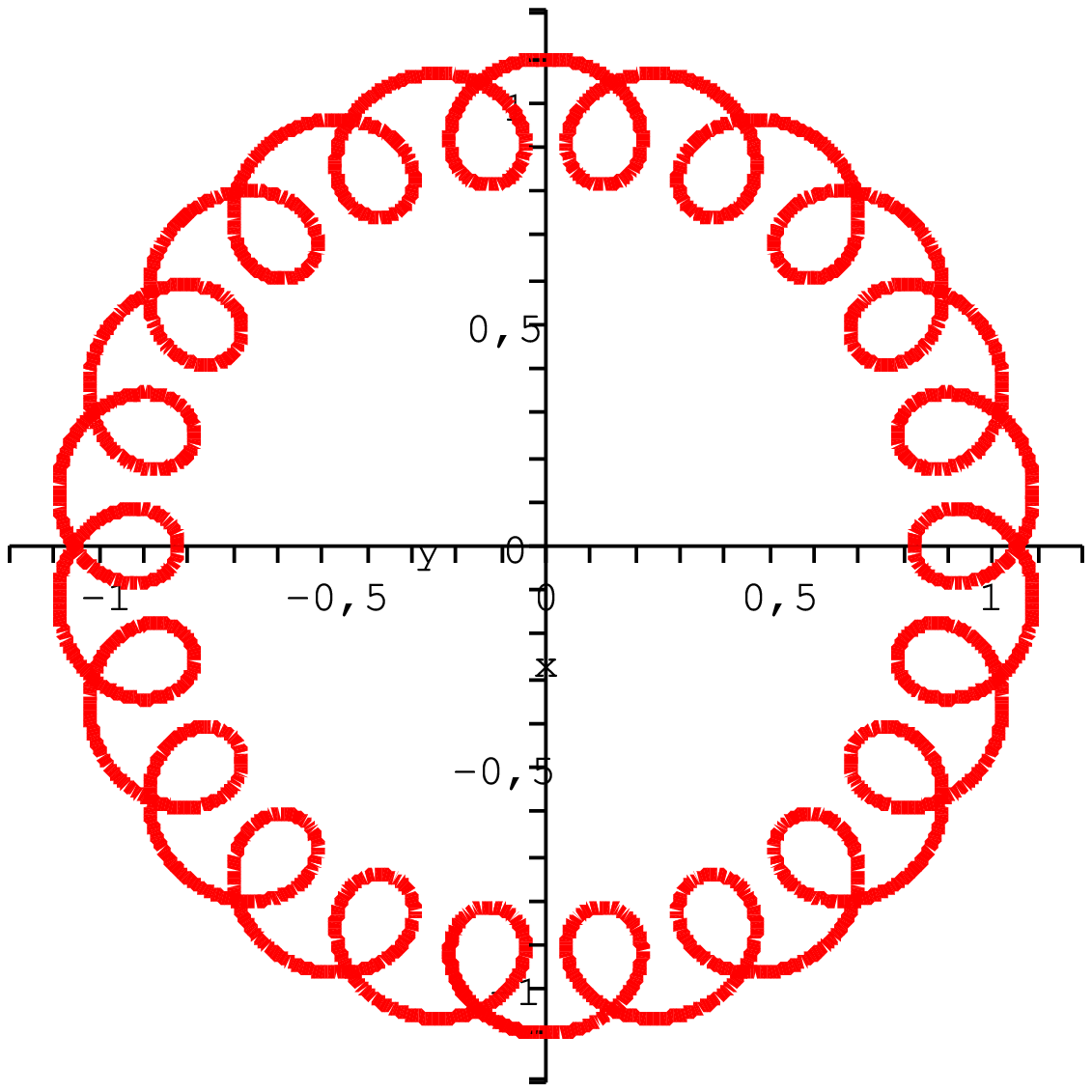}
   \caption{A closed curve $\al_p$ with  $n=3$}
\end{center}  \end{figure}

\begin{figure} \begin{center} \includegraphics[scale=0.5]{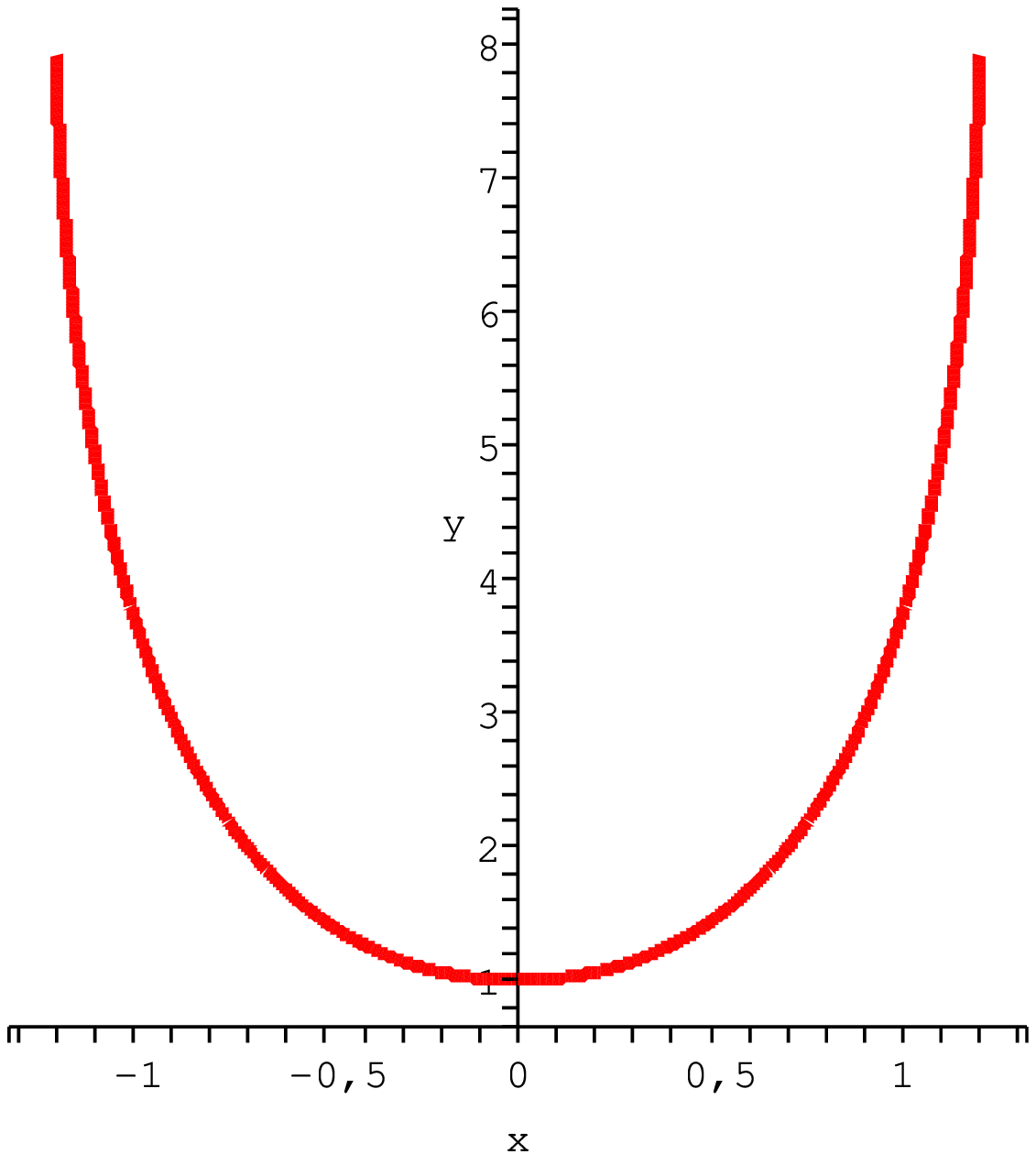}
   \caption{An unbounded curve with $n=3$}
\end{center}  \end{figure}

\noindent {\bf Case $n>2$.} Beyond the circles centered at the
origin, the solutions of the equation
$|\alpha|^{2n-2}\kappa_{\alpha^n}=constant$
 belong to one of the
following families:
\begin{itemize}
\item[-] A one-parameter family of non-embedded
curves, including a countable family of closed curves that we shall denote by $\al_p, p \in \Bbb N.$ The other
ones are not properly embedded;
\item[-] A one-parameter family of unbounded curves, some of them
being embedded and other ones not;
\item[-] A
one-parameter family of  curves which are not properly embedded,
with a spiraling end asymptotic to a circle centered at the
origin;
\item[-] A
one-parameter family of  curves with two spiraling ends asymptotic
to a circle centered at the origin;
\item[-]
Moreover, in the case of dimension $n=3,$ the straight lines (not
passing through the origin) are solutions.
\end{itemize}

\subsection{C-minimal Legendrian submanifolds in odd dimensional spheres}
We now describe several examples of C-minimal Legendrian immersions
into $\S^{2n-1}$, making special emphasis in the
compact case. We distinguish several cases according to the
values of $n$.

\subsubsection{\bf Case $n=2$.}
In this case, the Laplacian operator on a Legendrian curve $\gamma$ in $\S^3$
is simply the second derivative with respect to the arclength parameter $s$.
 Since $\frac{d\beta_\gamma}{ds} = k_\gamma$,
we deduce $\gamma$ is C-minimal if and only if it has constant curvature $k_\gamma=c.$

Such curves take the following explicit form (see \cite{CCh}):
$$
\gamma(s)=e^{i(c+\sqrt{c^2+4})s/2}A_1+e^{i(c-\sqrt{c^2+4})s/2}B_1,
$$
for suitable constants $A_1,B_1\in {\C}^2$. Since
$\Phi(t,s)=\alpha(t)\gamma(s)$, up to rotations it is enough to
consider the Legendrian curves parametrized by
\begin{equation}\label{levparalel}
\gamma_\varphi(s)= \left(\cos \varphi \, e^{i \tan \varphi \, s},
\sin \varphi \, e^{-i \cot \varphi \, s}\right), \  \varphi \in
(0,\pi/2),
\end{equation}
where $\pi/2 - 2\varphi $ is the latitude of the parallel $\pi
\circ \gamma_\varphi $. It is straightforward to check that the
Legendrian angle of $\gamma_{\varphi}$ is 
$$\beta_{\gamma_{\varphi}}=-\pi/2
+(\tan \varphi -\cot \varphi )s,$$
 hence it has constant curvature $c=\tan \varphi -\cot
\varphi$. Moreover the curve is closed if
and only if $\tan^2 \varphi=p/q \in \Q$. These curves are the links of
the cones whose Hamiltonian stability is studied in Theorem 7.1 of \cite{SW}.

\subsubsection{\bf Case $n=3$.}

There is a Legendrian immersion of the torus $\S^1 \times \S^1$ into $\S^5$ which is flat and minimal:
$\psi(s,t):=(e^{is}, e^{it}, e^{-i(s+t)}).$ Besides this trivial example,
there is a growing literature about C-minimal Legendrian surfaces into $\S^5$ and
H-minimal Lagrangian surfaces into $\C\P^2$. In particular, the existence of many  minimal and H-minimal Lagrangian tori immersed in
 $\C\P^2$ is proved respectively in \cite{HR3} in \cite{CM}, while a family of H-minimal Lagrangian tori in $\C\P^2$ with
$\S^1$-symmetry is explicitly described in \cite{MaSc} and \cite{M1}. However, as discussed in Section 2.1,
 the topological type of the Legendrian lifts are no longer
 the torus a priori, but rather its universal covering, i.e.\ the plane.

\subsubsection{\bf Case $n\geq 4$.} Besides the trivial example of the totally geodesic Legendrian embedding of
$\S^{n-1}$ into $\S^{2n-1}$, we point out that
compact minimal Legendrian immersions of $\S^1 \times \S^{n-2}$ into $\S^{2n-1}$ have been constructed in \cite{A2}. On 
the other hand a construction of C-minimal Legendrian immersions (including non-minimal ones) has been 
described in Section 3 of \cite{CLU} as follows: let $n_1,n_2$ be two integer numbers such that $n_1+n_2=n,$
$\gamma=(\gamma_1, \gamma_2): I \rightarrow \S^3$
 a Legendrian curve which is solution of one of the one-parametre family of o.d.e.
\begin{equation}
\label{eq:gammamu}
(\gamma_j' \overline{\gamma_j})(t) =
 (-1)^{j-1} i\, e^{i\mu t} \,    \overline{\gamma_1}(t)^{n_1} \,
\overline{\gamma_2}(t)^{n_2}, \,   \mu \in \R ,  \, j=1,2,
\end{equation}
and  $\psi_1$ and $\psi_2$
two C-minimal  Legendrian immersions of orientable manifolds $N_1$ and $N_2$ of dimensions 
$n_1-1$ and $n_2-1$
into $\S^{2n_1-1}$ and $\S^{2n_2-1}$ respectively. 
Then the following immersion
\[
\begin{array}{ccc}
\Psi: I \times N_1\times N_2 & \longrightarrow &
\S^{2n-1}, 
\\  \\
(t,x,y) &  \mapsto & ( \gamma_1 (t) \,\psi_1 (x) \,  , \, \gamma_2(t)
\,\psi_2 (y)),
\end{array}
\]
is C-minimal.

\medskip

While it turns out to be difficult to describe the general solution of (\ref{eq:gammamu}), and
therefore to control its closedness, we point out that the 
Legendrian curve with constant curvature $c=\sqrt{n_2/n_1}-\sqrt{n_1/n_2}$ (cf Section 4.3.1)

\[
\gamma_{n_1,n_2}(s)= \frac{1}{\sqrt{n_1+n_2}}\, \left(\sqrt
{n_1}\,e^{i \sqrt \frac{n_2}{n_1} \,s},
      \sqrt {n_2}\,e^{-i \sqrt \frac{n_1}{n_2} \,s}\right)
\]
is a closed solution of Equation (\ref{eq:gammamu}) with parametre $\mu =0.$

Hence, given two compact C-minimal Legendrian immersions of two manifolds $N_1$ and $N_2$ 
 into $\S^{2n_1-1}$ and $\S^{2n_2-1}$ respectively, we are able
to construct a compact immersion of C-minimal immersion of
$\S^1 \times N_1 \times N_2$ into $\S^{2n-1},$ where $n=n_1+n_2.$ 
We may start with simple examples such as  totally geodesic embedding, and iterate the process
in order to get immersions of a number of topological types:

\begin{corollary} \label{coroleg} Given $n$ integer numbers $m_1,...,m_n,$
there exist C-minimal immersions of $\T^{n} \times \prod_{j=1}^n \S^{m_j}$ into $\S^{2N-1}$, 
where $ N = n+\sum_{j=1}^n m_j$.
\end{corollary}

 \section{Construction with a Lagrangian surface and two Legendrian immersions}

\subsection{The construction}
In this section we give a detailed study of the geometry of a family of Lagrangian submanifolds
 introduced in \cite{CU3}.

\begin{theorem}\label{Th:Ex2}
Let $\phi=(\phi_1,\phi_2)$ be a Lagrangian
immersion of an orientable surface $\Sigma$ into $\C^2$ such that $\phi_1$ and
$\phi_2$ do not vanish, and
$\psi_1,\psi_2$  two Legendrian
immersions of orientable manifolds $N_1$ and $N_2$ into $\S^{2n_1-1}$ and $\S^{2n_2-1}$, with $n_1+n_2=n.$ Then the
map
$$ \begin{array}{lccc} \Phi :
 &  \Sigma \times N_1 \times N_2 &\  \longrightarrow& \C^{n}              \\
                   &&& \\
&  (p,x,y) & \longmapsto & (\phi_1(p) \psi_1(x),\phi_2(p) \psi_2(y))
   \end{array}$$
is a Lagrangian immersion with induced metric
\[
\bar{g}= g + |\phi_1|^2 g_1 + |\phi_2|^2 g_2,
\]
where $g,g_1$ and $g_2$ are the induced metric on $\Sigma, N_1$ and $N_2$ respectively,
 and whose Lagrangian angle map is
\begin{equation} \label{Beta2}
 \beta_{\Phi}(p,x,y) =(n_1-1) \pi +  G_{\!\phi}(p) + \beta_{\psi_1}(x) + \beta_{\psi_2}(y),
 \end{equation}
where $\beta_{\psi_1}$ and $\beta_{\psi_2}$ are the Legendrian angle maps of $\psi_1$ and $\psi_2$ and
 $G_{\!\phi}$ is the $\R / 2\pi \Z$-valued map defined on $\Sigma$ by
\[
 G_{\!\phi} := \beta_\phi + (n_1-1)\, \arg \phi_1 +(n_2-1)\, \arg\phi_2.
\]

 In addition the gradient of $\beta_{\Phi}$ is given by
\begin{equation}\label{Nabla2}
\nabla_{\!\bar{g}} \beta_{\Phi}=
 \left( \nabla_{\!{g}} G_{\!\phi},
  |\phi_1|^{-2}\nabla_{\!g_1} \beta_{\psi_1},  |\phi_2|^{-2}\nabla_{\!g_2} \beta_{\psi_2} \right).
  \end{equation}
  Moreover, given $X$, $Y$ and $Z$ three vector fields on $\Sigma$, $N_1$ and $N_2$ respectively, and
   denoting by $\bar{D}$ (resp.\ by $D$, $D^{1}$, $D^{2}$) the Levi-Civita connection of $\bar{g}$ (resp.\ of $g,g_1,g_2$),
     we have
\begin{equation}\label{Hess12}
 \bar{D}_{(X,0,0)}\nabla_{\!\bar{g}} \beta_{\Phi}=
 \left(D_X  \nabla_{\!g} G_{\!\phi},
  X( |\phi_1|^{-2}) \nabla_{\!g_1} \beta_{\psi_1}, X( |\phi_2|^{-2})\nabla_{\!g_2} \beta_{\psi_2} \right),
\end{equation}

\begin{equation}\label{Hess22}
 \bar{D}_{(0,Y,0)}\nabla_{\!\bar{g}} \beta_{\Phi}=
 \left( \frac{\nabla_{\!g} |\phi_1|^{-2}}{|\phi_1|^2}  d\beta_{\psi_1}(Y),
   \frac{D^{1}_Y \nabla_{\!g_1} \beta_{\psi_1}}{|\phi_1|^2} + g(\nabla_{\!g} G_{\!\phi}, \nabla_{\!g} |\phi_1|^{-2})Y, 0 \right),
 \end{equation}

 \begin{equation}\label{Hess32}
 \bar{D}_{(0,0,Z)}\nabla_{\!\bar{g}} \beta_{\Phi}=
 \left( \frac{\nabla_{\!g} |\phi_2|^{-2}}{|\phi_2|^2}  d\beta_{\psi_2}(Z), 0,
  \frac{D^{2}_Z \nabla_{\!g_2} \beta_{\psi_2}}{|\phi_2|^2} +  g(\nabla_{\!g} G_{\!\phi}, \nabla_{\!g} |\phi_2|^{-2}) Z \right)
\end{equation}

and
\begin{equation}\label{Delta2}
\Delta_{\bar{g}} \beta_\Phi = \frac{1}{|\phi_1|^{n_1-1}
|\phi_2|^{n_2-1}} \mbox{\rm div}_g
 \left(|\phi_1|^{n_1-1} |\phi_2|^{n_2-1}
  \nabla_{\!g} G_{\!\phi} \right) + \frac{\Delta_{g_1} \psi_1}{|\phi_1|^{2} }
      + \frac{\Delta_{g_2} \psi_2}{|\phi_2|^{2} }.
\end{equation}
\end{theorem}

\begin{proof}
  For the computation of the Lagrangian angle of the immersion $\Phi,$ we refer to \cite{CU3}.

 Let $(t_\mu), \mu=1,2$, $(x_a), 1 \leq a \leq n_1-1$ and $(y_j), 1
\leq j \leq n_2-1$ be local coordinates on $\Sigma$, $N_1$
and $N_2$ respectively. We shall denote by $(z_\alpha)=(t_\mu,x_a,y_j)$ the
resulting coordinates on $\Sigma \times N_1 \times N_2.$ It is
straightforward that
$$\bar{g}= g + |\phi_1|^2 g_{1} + |\phi_2|^2
g_{2}.$$

Thus, using that the matrix of $\bar g$ in the coordinates
$(z_\al)$ is block-diagonal, we compute:
$$ \nabla_{\!\bar{g}} \beta_\Phi = \sum_{\al,\beta=1}^{n}
   \bar{g}^{\al \beta}
               \frac{\pa \beta_\Phi}{\pa z_\beta}  \pa z_\al$$
$$=\sum_{\mu,\nu=1}^{2}   {g}^{\mu \nu}      \frac{\pa G_{\!\phi}}{\pa t_\mu}  (\pa_{t_\nu},0,0)
               +|\phi_1|^{-2} \sum_{a,b=1}^{n_1-1}
   {g}^{ab}               \frac{\pa \beta_{\psi_1}}{\pa x_b} (0,\pa_{x_a},0)$$
   $$+           |\phi_2|^{-2}\sum_{j,k=1}^{n_2-1}
   {g}^{jk}               \frac{\pa \beta_{\psi_2}}{\pa y_k} (0,0,\pa_{y_j}),$$
$$ =(\nabla_{\!{g}} G_{\!\phi},0,0)+
  |\phi_1|^{-2}(0,\nabla_{\!g_1} \beta_{\psi_1},0)+  |\phi_2|^{-2}(0,0,\nabla_{\!g_2} \beta_{\psi_2}).$$
 Next, a long but straightforward computation gives
$$\bar{D}_{\pa_{t_\mu}} \pa_{t_\nu} =  D_{\pa_{t_\mu}} \pa_{t_\nu},  \quad \quad
\bar{D}_{\pa_{x_a}} \pa_{t_\mu}=\bar{D}_{\pa_{t_\mu}} \pa_{x_a} =\frac{\pa (|\phi_1|^{-2})}{\pa t_\mu} \pa_{x_a},$$
$$\bar{D}_{\pa_{y_j}} \pa_{t_\mu}=\bar{D}_{\pa_{t_\mu}} \pa_{y_j} =\frac{\pa (|\phi_2|^{-2})}{\pa t_\mu} \pa_{y_j},$$
$$\bar{D}_{\pa_{x_a}} \pa_{x_b} = \nabla_{\!g} (|\phi_1|^{-2}) g_{ab} + D^1_{\pa_{x_a}} \pa_{x_b},$$
$$\bar{D}_{\pa_{y_j}} \pa_{y_k} =\nabla_{\!g} (|\phi_2|^{-2})  g_{jk} + D^2_{\pa_{y_j}} \pa_{y_k}.$$
From these equations it not difficult to obtain Equations (\ref{Hess12}), (\ref{Hess22}) and (\ref{Hess32}).
It remains to compute the Laplacian of $\beta_\Phi$:

$$\Delta_{\bar{g}} \beta_\Phi =\sum_{\al ,\beta=1}^{n}
 \bar{g} (\bar{D}_{\pa_{z_\al}} \nabla_{\!\bar{g}} \beta_\Phi, \pa_{z_\beta})\bar{g}^{\al \beta}$$
$$  = \sum_{\mu,\nu=1}^2
              \bar{g} (\bar{D}_{\pa_{t_\mu}} \nabla_{\!\bar{g}} \beta_\Phi, \pa_{t_\nu})\bar{g}^{\mu \nu}+ \sum_{a,b=1}^{n_1-1}
              \bar{g} (\bar{D}_{\pa_{x_a}} \nabla_{\!\bar{g}} \beta_\Phi, \pa_{x_b})\bar{g}^{ab}
+  \sum_{j,k=1}^{n_2-1} \bar{g} (\bar{D}_{\pa_{y_j}} \nabla_{\!\bar{g}} \beta_\Phi, \pa_{y_k})\bar{g}^{jk}$$
$$  = \sum_{\mu,\nu=1}^2
              g (\bar{D}_{\pa_{t_\mu}} \nabla_{\!\bar{g}} G_{\!\phi}, \pa_{t_\nu})g^{\mu \nu}+ \sum_{a,b=1}^{n_1-1}
              g_1 (\bar{D}_{\pa_{x_a}} \nabla_{\!\bar{g}} \beta_\Phi, \pa_{x_b})g^{ab}
+  \sum_{j,k=1}^{n_2-1} g_2 (\bar{D}_{\pa_{y_j}} \nabla_{\!\bar{g}} \beta_\Phi, \pa_{y_k})g^{jk}$$
$$  = \Delta_g G_{\!\phi}+ \sum_{a,b=1}^{n_1-1}
            \left( \frac{1}{|\phi_1|^{2}} g_1 (D^{1}_{\pa_{x_a}} \nabla_{\!g_1} \beta_{\psi_1}, \pa_{x_b} )
            +g(\nabla_{\!g} G_{\!\phi}, \nabla_{\!g} |\phi_1|^{-2}) g_{ab}   \right)g^{ab}$$
              $$
+  \sum_{j,k=1}^{n_2-1}
\left( \frac{1}{|\phi_2|^{2}} g_2 (D^{2}_{\pa_{y_j}} \nabla_{\!g_2} \beta_{\psi_2}, \pa_{y_k})
 +g(\nabla_{\!g} G_{\!\phi}, \nabla_{\!g} |\phi_2|^{-2}) g_{jk}\right)    g^{jk}$$
$$= \Delta_g G_\phi+ \frac{1}{|\phi_1|^2} \Delta_{g_1} \beta_{\psi_1}+ (n_1-1) g(\nabla_{\!g} G_{\!\phi}, \nabla_{\!g} |\phi_1|^{-2})$$
$$+
\frac{1}{|\phi_2|^2} \Delta_{g_2} \beta_{\psi_2}+ (n_2-1) g(\nabla_{\!g} G_{\!\phi}, \nabla_{\!g} |\phi_2|^{-2}).$$

Finally, we check that
$$ \frac{1}{|\phi_1|^{n_1-1}
|\phi_2|^{n_2-1}} \mbox{\rm div}_{\!g}
 \left(|\phi_1|^{n_1-1} |\phi_2|^{n_2-1}
  \nabla_{\!g} G_{\!\phi} \right) $$
 $$ =\Delta_g G_{\!\phi} + (n_1-1) g(\nabla_{\!g} G_{\!\phi}, \nabla_{\!g} |\phi_1|^{-2})
 +(n_2-1) g(\nabla_{\!g} G_{\!\phi}, \nabla_{\!g} |\phi_2|^{-2}),$$
which completes the proof.

\end{proof}

\begin{remark}
{\rm

If the Lagrangian immersion $\phi $ is the product of two planar curves $(\al_1,\al_2),$
 then $\Phi $ is the product of two Lagrangian immersions $\Phi_1$ and $\Phi_2$ of the type
 described in  Theorem \ref{Th:Ex1}. This fact will be important for the classification of
 immersions $\Phi$ with parallel mean curvature vector (Corollary \ref{Cor:Ex2}).

If the Lagrangian immersion is of the type of Theorem
\ref{Th:Ex1}, i.e.\ $\phi =\alpha \gamma $ where $\al$ is a planar
curve and $\gamma$ is a Legendrian curve in $\S^3$, then the
immersion $\Phi $ is also of the type of Theorem \ref{Th:Ex1},
with $\psi=(\gamma_1 \psi_1,\gamma_2 \psi_2)$ is of the type
described in \cite{CLU} (see also Section 4.3.3). Moreover, any
Lagrangian immersion in $\C^n$ invariant under the action of
$SO(n_1)\times SO(n_2)$, with $n_1+n_2=n$ and $n_1,n_2\geq 2$, is
congruent to an open subset of one of the Lagrangian submanifolds
of Theorem \ref{Th:Ex2}, where $\psi_1(x)=x$ and $\psi_2(y)=y$ are
the totally geodesic embeddings of $\,\S^{n_i-1}$ into
$\,\S^{2n_i-1}$, $i=1,2$. }
\end{remark}

The  Lagrangian immersions described in Theorem \ref{Th:Ex2} have
singularities at the points $(p,x,y)\in \Sigma\times N_1\times
N_2$ where either $\phi_1(p)=0$ or $\phi_2(p)=0$.

It has been proved in \cite{CU3} (and it follows from Equation (\ref{Beta2})) that the immersion
$\Phi $ described in Theorem \ref{Th:Ex2} is minimal if and only
if $\psi_1$ are $\psi_2$ are minimal and $G_\phi $ is constant. If
$n=3$, then either $(n_1,n_2)=(2,1)$ or $(1,2)$. In
both cases, $\psi_1$ or $\psi_2$ must be a Legendrian geodesic in
$\S^3$, which can be parameterized, up to congruence by
$\psi(t)=\frac{1}{\sqrt 2}(e^{it},e^{-it})$. The corresponding
Lagrangian immersion $\Phi$ takes the following form:
\[
\begin{array}{c}
\Phi:\Sigma\times\R\longrightarrow\C^3\\
(p,t)\mapsto\left(\frac{\phi_1(p)}{\sqrt
2}e^{it},\frac{\phi_1(p)}{\sqrt 2}e^{-it},\phi_2(p)\right).
\end{array}
\]
These examples, which are invariant under the action of $U(1)\equiv SO(2)$ on $\C^3$ given by
\[
e^{it}\cdot(z_1,z_2,z_3)=(e^{it}z_1,e^{-it}z_2,z_3)
\]
  have been studied in detail  by Joyce in \cite{J2} in the
minimal Lagrangian case.

\medskip

Before we  characterize  those immersions $\Phi$ which have parallel mean curvature vector  or are compact and H-minimal, we need
 a technical lemma:

\begin{lemma} \label{lag}
Let $\Sigma$  be a connected, properly immersed Lagrangian
 surface of $\C^2$ with coordinates $(\phi_1,\phi_2)$ and such that  the vector fields  $\nabla_{\!g} |\phi_1|$ and $\nabla_{\!g} |\phi_2|$ are linearly dependent.
  Then either $\Sigma$ is contained in a product of two planar curves $\al_1$ and $\al_2$ (one of them being  an arc of circle), or
  it may be locally parametrized 
  by an immersion of the form
  $$\phi(s,t)=(\sqrt{a x(s)+b} \, e^{i(s+t)},\sqrt{x(s)} e^{i(s-a t)}).$$
  where $x(s)>0$ and $(a,b) \in \Q \times \R.$
\end{lemma}

\begin{proof} Suppose first that $|\phi_1|$ is constant.
We proceed by contradiction assuming that $\Sigma$ is not a product of curves. Then it
may be locally parametrized by
$$ \phi(s,t)=( C e^{i \theta(s,t)},s+it),$$
where $s+it$ belongs to an open subset of $\C$, $C$ is a positive constant and $\theta$ an $\R / 2\pi \Z$-valued map.
If follows that $\omega(\phi_s,\phi_t)=1$, which contradicts the Lagrangian assumption. Of course the argument is the same if 
$|\phi_2|$ is constant.

It remains to treat the case in which $|\phi_1|$ and $|\phi_2|$ are not constant. There exists therefore a smooth real map $u$
such that $|\phi_1|^2 = u(|\phi_2|^2).$ Introduce the hypersurface
$$ Q :=\{ (z_1,z_2) \in \C^2, |z_1|^2 - u (|z_2|^2) = 0 \}.$$
Since $\Sigma \subset Q, $ we have 
$ T^\perp Q \subset T^\perp \Sigma.$ 
 On the other hand, 
at a point $(z_1,z_2)$ of $Q$, $T^\perp_{(z_1,z_2)} Q = (z_1, -u'(|z_2|^2) z_2) \R,$ thus  
$(z_1,-u'(|z_2|^2) z_2)\in T^\perp_{(z_1,z_2)} \Sigma.$
By the Lagrangian assumption, the vector field $V(z_1,z_2)=(iz_1, -iu'(|z_2|^2) z_2)$ is tangent to $\Sigma$, so its integral curves
 $$t \mapsto (z_1^0 e^{it},z_2^0 e^{-i u'(|z_2^0|^2) t}),$$
  with initial point $(z_1^0, z_2^0) \in \Sigma,$ are contained in $\Sigma.$
 It follows that $\Sigma$ may be locally parametrized by
$$ \phi(s,t)=(z_1(s) e^{it},z_2(s) e^{-iu'(|z_2(s)|^2) t}),$$
where $(z_1(s),z_2(s))$ is a curve in $\Sigma$ tranversal to the integral curves of the vector field $V.$
Moreover, we may replace the curve $(z_1(s),z_2(s))$ by $(\tilde{z}_1(s),\tilde{z}_2(s))=(z_1(s) e^{i \theta(s)}, z_2(s) e^{-iu'(|z_2(s)|^2) \theta(s)})$
without changing the image of $\phi$. A routine computation shows that we can choose the function $\theta(s)$ in order to have 
  $\arg \tilde{z}_1 = \arg \tilde{z}_2.$ 
 Using the assumption $(\tilde{z}_1(s),\tilde{z}_2(s)) \in Q,$ we deduce that 
 $\tilde{z}_1(s)=  \sqrt{u(|\tilde{z}_2(s)|^2)}\tilde{z}_2(s)$ so, writing $\tilde{z}_2(s)= \sqrt{x(s)} e^{is},$ we have
 $\tilde{z}_1(s)= \sqrt{u(x(s))} e^{is}$ and we obtain
$$\phi(s,t)=(\sqrt{u(x(s))} e^{i(s+t)},\sqrt{x(s)} e^{i(s-u'(x(s)) t)}).$$
We claim now that the properly immersed assumption imposes a restriction on $u$:
 if $u'(|z_2^0|^2)$ does not belong to $ \Q,$ then the integral curves of $V$ are not closed; more precisely they are
 dense in the torus $\{(z_1, z_2) \in \C^2 \big| \, \, |z_1|= |z_1^0|, |z_2|=|z_2^0| \}.$ Since $\Sigma$ is assumed to be properly immersed, 
 it must be contained in this torus, which is a product of curves. If it is not the case,
  $u'(|z_2^0|^2)$  must be rationaly related to $\Q$,  so it is constant and $u(x)=a x + b,$ where $(a,b) \in \Q \times \R$, which implies the
 claimed formula.
\end{proof}

\begin{corollary}\label{Cor:Ex2}
Consider the Lagrangian immersion
$$ \begin{array}{lccc} \Phi :
 &  \Sigma \times N_1 \times N_2 &\  \longrightarrow& \C^{n}              \\
                   &&& \\
&  (p,x,y) & \longmapsto & (\phi_1(p) \psi_1(x),\phi_2(p) \psi_2(y))
   \end{array}$$
 where $\phi=(\phi_1,\phi_2):\Sigma\rightarrow\C^2$ is a  Lagrangian
immersion of an orientable, properly immersed surface $\Sigma$ with Lagrangian angle $\beta_\phi$ and
$\psi_i,$ $i=1,2$ are two Legendrian
immersions of orientable manifolds $N_i$ into $\S^{2n_i-1},$ where $n_1+n_2=n.$
 Then $\Phi$
 has parallel mean curvature vector if and only if
\begin{itemize}
    \item[-]  either (i) it is minimal; in this case the immersion $\phi$ is such that
$G_{\!\phi}:=\beta_\phi + (n_1-1)\, \arg \phi_1 +(n_2-1)\, \arg\phi_2$ is constant and the immersions
$\psi_1$ and $\psi_2$ are minimal (see \cite{CU3});
\item[-] or (ii) it is a product $\Phi=(\Phi_1,\Phi_2)$ of two Lagrangian immersions into $\C^{n_1}$ and $\C^{n_2}$
with parallel mean curvature vector (as described in Theorem \ref{Th:Ex1});  in this case the immersion $\phi$ is a product
of two planar curves $(\al_1,\al_2)$;
\item[-] or (iii) $\phi$ is, up to scaling, the immersion
$$\phi(s,t)=(\cos s \, e^{it}, \sin s \, e^{it})$$ 
 and the immersions
$\psi_1$ and $\psi_2$ are minimal.
\end{itemize}
  Moreover, if $N_1$ and $N_2$ are compact, the immersion $\Phi$ is H-minimal
  if and only if
   the immersion $\phi$ is such that the vector field
   $$A_\phi:=|\phi_1|^{n_1-1} |\phi_2|^{n_2-1}
  \nabla_{\!g} G_{\!\phi}$$
   is divergence free and
the immersions $\psi_1$ and $\psi_2$ are  C-minimal.
\end{corollary}

\begin{proof}

Assume that $|\phi_1|$ and $|\phi_2|$ are not constant. The fact that $\nabla_{\!\bar{g}}\beta_\Phi$ is parallel and
Equations (\ref{Hess12}), (\ref{Hess22}) and (\ref{Hess32}) of Theorem \ref{Th:Ex2} imply that $\beta_{\psi_1}$ and
$\beta_{\psi_2}$ are constant and
that
 \begin{equation} \label{ps1}
 g(\nabla_{\!g} G_{\!\phi}, \nabla_{\!g} |\phi_1|^{-2})=0, 
 \end{equation}
  \begin{equation} \label{ps2}
  g(\nabla_{\!g} G_{\!\phi}, \nabla_{\!g} |\phi_2|^{-2})=0.
 \end{equation}
Hence, if the vector fields  $\nabla_{\!g} |\phi_1|^{-2}$ and $\nabla_{\!g} |\phi_2|^{-2}$ are independent, the vector field
      $\nabla_{\!g} G_{\!\phi}$  vanishes, so  the map $G_{\!\phi}$ is constant. 
    By Equation (\ref{Beta2}) of Theorem \ref{Th:Ex2}, it follows that the immersion
      $\Phi$ is minimal and we are in case \em(i). \em 
We now assume that $G_\phi$ is not constant and that the vector fields  $\nabla_{\!g} |\phi_1|^{-2}$ and $\nabla_{\!g} |\phi_2|^{-2}$ are linearly dependent. Using  Lemma \ref{lag}, and since $|\phi_1|$ and $|\phi_2|$ are not constant, the immersion 
 $\phi$ takes the form
 $$\phi(s,t)=(\sqrt{a x(s)+b} \,  e^{i(s+t)},\sqrt{x(s)} \, e^{i(s-a t)}),$$
 where $x(s) >0$ is not constant and $(a,b) \in \Q \times \R.$
We easily compute that 
  $$ G_\phi = (n_1+n_2)s + (n_1 - an_2) t -\arctan \left(\frac{ x'((1+a)ax + b)}{2(1+a)x(ax+b)} \right)$$
  and use the assumption that $\nabla^g G_{\phi}$ is parallel (coming from Equation (\ref{Hess12})), more precisely, the vanishing
  of $g(\nabla^g_{\pa_s} G_\phi, \pa_t) $ and $g(\nabla^g_{\pa_t} G_\phi, \pa_t) .$
    A tedious but straightforward computation involving the Christoffel symbols of the induced metric $g$ shows that either  $a=-1$ or $b$  vanishes. 
If $b$ vanishes  the immersion $\phi$ takes the simpler form
$$ \phi(s,t)=(\sqrt{a} \, \alpha(s) e^{it}, \alpha(s) e^{-iat}),$$
where $\al:= \sqrt{x(s)} e^{is}.$
Hence the immersion $\Phi$ can be written as $\Phi(s,t,x,y)=\sqrt{1+a} \, \al(s) \psi(t,x,y)$, where
$$\psi(t,x,y):=\frac{1}{\sqrt{1+a}}(\sqrt{a} e^{it} \psi_1(x), e^{-ia t} \psi_2(y))$$
is a Legendrian immersion into $\S^{2n-1}.$
In other words, we are in the case of Section 4, and we can rely on the Corollary \ref{Cor:Ex1} to deduce that either the immersion
 $\Phi$ is minimal, which is case \em(i), \em unless the curve $\alpha$ is a circle. However this last case is excluded since $x$ is assumed to
 be not constant.
If 
 $a=-1,$ setting $\sigma:= \arcsin \sqrt{x/b}$ and $\tau :=s+t,$ the immersion $\phi$ becomes
 $$ \phi(\sigma,\tau)=(b \cos \sigma e^{i\tau},b \sin \sigma e^{i\tau}),$$
 so we are in case \em (iii). \em
  Finally, if one of the quantities $|\phi_1|$ and $|\phi_2|$ is constant, by the second point of Lemma \ref{lag}, the immersion
      $\phi$ is a product of planar curves and we get the case \em (ii). \em

To complete the proof, we observe that if $\Phi$ is H-minimal, by Equation (\ref{Delta2}),
$\Delta_{g_1} \beta_{\psi_1}$ and $\Delta_{g_2} \beta_{\psi_2}$ must be constant. By the compactness assumption, those
constants must be zero, so the immersions $\psi_1$ and $\psi_2$ are C-minimal and the immersion $\phi$ is such that
the vector field $|\phi_1|^{n_1-1} |\phi_2|^{n_2-1}
  \nabla_{\!g} G_\phi$ is divergence free.
\end{proof}

In Section 4.3 we have given a description of
several compact, C-minimal  Legendrian immersions in odd-dimensional
spheres. On the other hand, it turns out to be  difficult to solve
the equation $ \mbox{div }A_\phi=0$ in full generality. In the next
section we make
use  of a construction introduced in \cite{CCh} in order to get a countable family of solutions
to this equation.

\subsection{A special class of Lagrangian surfaces of $\C^2$ satisfying  $ \mbox{div}A_\phi=0$}
We follow the notation of Section 2.3 and consider
  a Legendrian curve $\gamma :$
$I_1 \to \S^3$ and  a
Legendrian curve $\alpha :$  $I_2 \to \H^3_1$, that we both assume to be parametrized by arclength.
Consider the map: $\phi:I_1\times I_2\subset {\R}^2\to {\C}^2
={\C}\times{\C}$
defined by
\[
\phi(s,t)\equiv  \gamma(s) \odot \alpha(t)=
(\gamma_1(s)\alpha_1(t),\gamma_2(s) \alpha_2(t)).
\]
Then $\phi=\gamma \odot \alpha$ is a Lagrangian conformal
immersion in ${\C}^2$ whose induced metric is given by
\begin{equation}\label{metric}
g=(|\gamma_1|^2
+|\alpha_1|^2)(ds^2+dt^2)=(\xi_3+\eta_3)(ds^2+dt^2).
\end{equation}
The Lagrangian angle map $\beta_\phi$ of the  Lagrangian conformal
immersion $\phi=\gamma \odot \alpha$ and the Legendrian angles
$\beta_\gamma$ and $\beta_\alpha$ of $\gamma$ and $\alpha$ are
related by
\begin{equation}\label{beta}
\beta_{\phi}(s,t)= \beta_{\gamma}(s)+\beta_{\alpha}(t)+\pi.\end{equation}
Setting $G\equiv G_{\!\phi}$  for sake of brevity, we have
\begin{equation}\label{G}
G = \beta_\gamma + (n_1-1) \arg \gamma_1 + (n_2-1)
\arg \gamma_2 + \beta_\alpha + (n_1-1) \arg \alpha_1 + (n_2-1)
\arg \alpha_2 + \pi.
\end{equation}
On the other hand,
$$ A_\phi=\frac{(|\gamma_1||\alpha_1|)^{n_1-1}(|\gamma_2||\alpha_2|)^{n_2-1}}{|\gamma_1|^2+|\alpha_1|^2}(G_s,G_t).$$
Hence
we deduce that
$$(|\gamma_1|^2+|\alpha_1|^2)\mbox{div }A_\phi $$
$$= \Big( (|\gamma_1||\alpha_1|)^{n_1-1}(|\gamma_2||\alpha_2|)^{n_2-1} G_s\Big)_s+ 
\Big((|\gamma_1||\alpha_1|)^{n_1-1}(|\gamma_2||\alpha_2|)^{n_2-1} G_t \Big)_t.$$
A particular solution arises when the two terms of the right hand side vanish, i.e.\ if there exist
two real constants $c_1$ and $c_2$ such that
$$|\gamma_1|^{n_1-1}|\gamma_2|^{n_2-1} G_s =c_1,$$
$$|\alpha_1|^{n_1-1}|\alpha_2|^{n_2-1} G_t = c_2.$$
Since we have
\begin{equation}\label{Gs}
G_s=\beta'_\gamma+\<\gamma_1',J\gamma_1\> \left( \frac{n_1-1}{|\gamma_1|^2} -
\frac{n_2-1}{|\gamma_2|^2} \right),
\end{equation}
\begin{equation}\label{Gt}
G_t=\beta'_\alpha+\<\alpha_1',J\alpha_1\> \left( \frac{n_1-1}{|\alpha_1|^2} -
\frac{n_2-1}{|\alpha_2|^2} \right),
\end{equation}
we obtain the following sufficient 
conditions  in order $\phi=\gamma
\odot  \alpha$ to satisfy $\mbox{ div } A_\phi=0$:
\begin{equation} \label{cond1}
\beta'_\gamma+\<\gamma_1',J\gamma_1\> \left( \frac{n_1-1}{|\gamma_1|^2} -
\frac{n_2-1}{|\gamma_2|^2} \right)\!=\! \frac{c_1}{
|\gamma_1|^{n_1-1} \, |\gamma_2|^{n_2-1}},
\end{equation}

\begin{equation}\label{cond2}
 \beta'_\alpha+\<\alpha_1',J\alpha_1\> \left( \frac{n_1-1}{|\alpha_1|^2} -
\frac{n_2-1}{|\alpha_2|^2} \right)\!=\! \frac{c_2}{
|\alpha_1|^{n_1-1}
 |\alpha_2|^{n_2-1}}.
\end{equation}

Finally, using the relations between the Legendrian curves $\gamma$ and $\alpha$
and their Hopf projections $\xi = \Pi \circ \gamma$ and $\eta = \Pi \circ \alpha$,
(cf Section 2.3) we get:
\begin{equation} \label{cond11}
k_\xi +(\xi \times \xi')_3 \left( \frac{n_1-1}{1/2+\xi_3} -
\frac{n_2-1}{1/2-\xi_3} \right)\!=\! \frac{c_1}{
(1/2+\xi_3)^{(n_1-1)/2} \, (1/2-\xi_3)^{(n_2-1)/2}},
\end{equation}

\begin{equation}\label{cond22}
k_\eta +(\eta \!\times\! \eta')_3 \left( \frac{n_1-1}{\eta_3-1/2}
+ \frac{n_2-1}{\eta_3+1/2} \right) \!=\! \frac{c_2}{
(\eta_3-1/2)^{(n_1-1)/2}
 (1/2+\eta_3)^{(n_2-1)/2}}.
\end{equation}
 The Legendrian lift $\gamma$ or $\alpha$ of a generic solution of Equation (\ref{cond11}) or 
(\ref{cond22}) is not expected to be closed, even if the projected curve $\xi$ or $\eta$ is so. Since we are
interested in compact examples, we point out that the
geodesic $\xi_0=\Pi\circ \gamma_0,$ where $\gamma_0(s)= (\cos s, \sin s
)$ trivially verifies (\ref{cond1}) with $c_1=0$ and the constant
curvature curve $\eta_\delta (t) = \Pi (\alpha_\delta (t))$, with
$$
\alpha_\delta(t)=\left(  \sinh \delta \, e^{i\coth \delta \, t},
\cosh \delta  \, e^{i  \tanh \delta \, t }\right),\  \delta
>0
$$
is a solution of (\ref{cond2}) for a suitable constant $c_2(\delta)$ depending on $\delta.$
 Moreover
$\alpha_\delta $ is a closed curve if and only if $\tanh^2 \delta
= q/r\in \Q$, $0<q<r$. In such a case $\alpha_\delta $ takes the form
$$
\alpha_{q,r}(t)= \frac{1}{\sqrt{r-q}}\left(  \sqrt q \,
e^{i\sqrt{r/q} \, t}, \sqrt{r} \, e^{i  \sqrt{q/r} \, t }\right),\
0<q<r.
$$
 It follows that the corresponding immersions
  $\phi_{q,r}:=\gamma_0 \odot \alpha_{q,r} $ are doubly periodic and
 their images are
 non-trivial Lagrangian tori. These tori, besides being solutions of
 $\mbox{ div } A_\phi=0$, turn out to be  H-minimal and self-similar for the mean curvature flow as well (cf \cite{CL}).
  
  \bigskip
  
 Summing up, 
   using the two constructions of H-minimal Lagrangian immersions (Corollaries \ref{Cor:Ex1} and \ref{Cor:Ex2}) and the existence result of C-minimal Legendrian immersions (Corollary \ref{coroleg}),  we are able to construct
families of  compact H-minimal Lagrangian immersions  of various topological types. In the first case, the  Lagrangian immersions 
$\Phi_p=\al_p \psi, p \in \Bbb N,$ constructed  using the countable family of closed curves $\al_p$ described in Section 4.2 and a $C$-minimal Legendrian immersion $\psi$ 
do not have parallel mean curvature vector  by Corollary \ref{Cor:Ex1}. Analogously,
   since the tori $\phi_{q,r}$ are not minimal
 and are not a Cartesian product of planar curves,  
 the immersions $\Phi_{q,r}$ do not have parallel mean curvature by Corollary  \ref{Cor:Ex2}.  Hence, we have got:

\begin{corollary} \label{corolag} Given $n$ integer numbers $m_1,...,m_n,$
there exist two countable families of compact immersions
$\Phi_p$ and $\Phi_{q,r}$ 
of $\T^{n} \times \prod_{j=1}^n \S^{m_j}$ into $\C^{N}$, 
where $ N = n+\sum_{j=1}^n m_j,$ which are H-minimal and non trivial, i.e.\
non minimal and whose mean curvature vector is not parallel.
\end{corollary}

\end{document}